\documentclass{amsart}
\usepackage{graphicx}
\usepackage[all]{xy}
\usepackage[ansinew]{inputenc}
\usepackage{amsmath}
\usepackage{amscd}
\usepackage{amssymb}
\usepackage{latexsym}
\usepackage[active]{srcltx}
\usepackage{mathrsfs}
\usepackage{color}
\usepackage{enumerate}
\usepackage{hyperref}

\newtheorem{thm}{Theorem}[section]

\theoremstyle{definition}

\newtheorem{defn}[thm]{Definition}

\newtheorem*{thmA}{Main result}

\numberwithin{equation}{section}

\newcommand{\N}{\mathbb{N}}
\newcommand{\Z}{\mathbb{Z}}
\newcommand{\Q}{\mathbb{Q}}
\newcommand{\R}{\mathbb{R}}

\newcommand{\dcl}{\operatorname{dcl}}

\newcommand{\rtil}{\tilde{\mathbb{R}}}

\newcommand{\Cal}{\mathcal}

\newcommand{\OR}{\overline{\R}}

\def \<{\langle}
\def \>{\rangle}

\def \((  {(\!(}
\def \)) {)\!)}

\begin{document}

\title[Distal and Non-Distal Pairs]{Distal and Non-Distal Pairs}

\author[P. Hieronymi]{Philipp Hieronymi}
\address
{Department of Mathematics\\University of Illinois at Urbana-Champaign\\1409 West Green Street\\Urbana, IL 61801}
\email{phierony@illinois.edu}
\urladdr{http://www.math.uiuc.edu/\textasciitilde phierony}

\author[T. Nell]{Travis Nell}
\address
{Department of Mathematics\\University of Illinois at Urbana-Champaign\\1409 West Green Street\\Urbana, IL 61801}
\email{tnell2@illinois.edu}
\urladdr{http://www.math.uiuc.edu/\textasciitilde tnell2}


\date{\today}

\begin{abstract}
 The aim of this note is to determine whether certain non-o-minimal expansions of o-minimal theories which are known to be NIP, are also distal. We observe that while tame pairs of o-minimal
structures and the real field with a discrete multiplicative subgroup have distal theories, dense pairs of o-minimal structures and related examples do not.
\end{abstract}

\thanks{A version of this paper is to appear in the \emph{Journal of Symbolic Logic}. The first author was partially supported by NSF grant DMS-1300402.}

\maketitle

\section{Introduction}

Over the last two decades, \emph{NIP} (or \emph{dependent}) theories, first introduced by Shelah in
\cite{Shelah}, have attracted substantial interest. Properties of these theories have been studied in detail, and many examples of such theories have been constructed (see \cite{Simon-Book} for a modern overview of the subject). Recently, Simon \cite{Simon-Distal} identified an important subclass of NIP theories called \emph{distal theories}. The motivation behind this new notion is to single out NIP theories that can be considered purely unstable. O-minimal theories, the classical examples of unstable NIP theories, are distal. The aim of this note is to determine whether certain non-o-minimal expansions of o-minimal theories which are known to be NIP, are also distal.\newline

Let $\Cal A=(A,<,\dots)$ be an o-minimal structure expanding an ordered group and let $B\subseteq A$.  We consider theories of structures of the form $(\Cal A,B)$ that satisfy one of the following conditions:
\begin{itemize}
\item[1.] $\Cal A$ is the real field and $B$ is a cyclic multiplicative subgroup of $\R_{>0}$ (\emph{discrete subgroup}),
\item[2.] $\Cal A$ expands a real closed field, $B$ is a proper elementary substructure such that there is a unique way to define a standard part map from $A$ into $B$ (\emph{tame pairs}),
\item[3.] $B$ is a proper elementary substructure of $\Cal A$ dense in $A$ (\emph{dense pairs}),
\item[4.] $\Cal A$ is the real field and $B$ is a dense subgroup of the multiplicative group of $\R_{>0}$ with the Mann property (\emph{dense subgroup}),
\item[5.] $B$ is a dense, definably independent set (\emph{independent set}).
\end{itemize}
Here and throughout this paper, dense means dense in the usual order topology on $A$. All the above examples are NIP. For dense pairs this is due independently to Berenstein, Dolich, Onshuus \cite{BDO}, Boxall \cite{Boxall}, and G\"unayd\i n and Hieronymi \cite{GH-Dependent}; for dense groups this was shown in \cite{Boxall} and \cite{GH-Dependent}; for tame pairs and for the discrete subgroups NIP was first proven in \cite{GH-Dependent}. For a later, but more general result implying NIP for all these theories, see Chernikov and Simon \cite{CS-Ext}. Our main results here are as follows.

\begin{thmA} The theories of structures satisfying 1. or 2. are distal. The theories of structures satisfying 3., 4., or 5. are not distal.
\end{thmA}

\noindent  We observe the following interesting phenomenon: All examples of the above NIP theories that do not define a dense and codense set, are distal. However, all the examples that define a dense and codense set, are not distal. This not true in general. The expansion $(\R,<,\Q)$ of the real line by a predicate for the set of rationals is dp-minimal and hence distal by \cite[Lemma 2.10]{Simon-Distal}.


\subsection*{Definitions and notations} Here are precise definitions of the properties under investigation.
Let $T$ be a complete theory in a language $\Cal L$ and let $\mathbb M$ be a monster model of $T$.  When a sequence
$(a_i)_{i\in I}$ from $\mathbb M^p$ is indiscernible over a parameter set $A$, we say the sequence is $A$-indiscernible. We assume that such a parameter set $A\subseteq \mathbb M$ always has cardinality smaller than the cardinality of saturation of $\mathbb M$. If we say a sequence is indiscernible, we mean the sequence is $\emptyset$-indiscernible.

\begin{defn} We call an $\Cal L$-formula $\varphi(x,y)$ {\it dependent} (in $T$) if for every indiscernible sequence
$(a_i)_{i\in \omega}$ from $\mathbb M^p$ and every $b\in\mathbb M^q$,
there is $i_0\in \omega$ such that either $\mathbb M \models \varphi(a_{i},b)$ for every $i>i_0$ or $\mathbb M \models \neg \varphi(a_i,b)$ for every $i>i_0$. The theory $T$ \emph{is NIP} (or \emph{is dependent}) if every $\Cal L$-formula is dependent in $T$.
\end{defn}

\noindent Here and in what follows, $I,I_1,I_2$ will always be linearly ordered sets. When we write $I_1+I_2$, we mean the concatenation of $I_1$ followed by $I_2$. By $(c)$ we denote the linearly ordered set consisting of a single element $c$.

\begin{defn}\label{defn:distal} We say $T$ is \emph{distal} if whenever $A \subseteq \mathbb M$, and $(a_i)_{i\in I}$ an indiscernible sequence from $\mathbb M^p$ such that
\begin{itemize}
\item[a.] $I = I_1 + (c) + I_2$, and both $I_1$ and $I_2$ are infinite without endpoints,
\item[b.] $(a_i)_{i\in I_1 + I_2}$ is $A$-indiscernible,
\end{itemize}
then $(a_i)_{i\in I}$ is $A$-indiscernible.
\end{defn}
\noindent 
It is an easy exercise to check that every distal theory as defined above is also NIP. When $T$ is NIP, the definition of distality given above is one of several equivalent definitions. Here we will only use this characterization of distality, and we refer the interested reader to \cite{Simon-Distal,Simon-Book} for more information.\newline

\noindent For the purposes of this paper it is convenient to introduce the following notion of distality for a single $\Cal L$-formula.

\begin{defn} Let $\varphi(x_1,\dots,x_n;y)$ be a (partitioned) $\Cal L$-forumula, where $x_i=(x_{i,1},\dots, x_{i,p})$ for each $i=1,\dots,n$. We say $\varphi(x_1,\dots,x_n;y)$ is \emph{distal} (in $T$) if for $b\in \mathbb M^q$ and every indiscernible sequence  $(a_i)_{i\in I}$ from $\mathbb M^p$ that satisfies
\begin{itemize}
\item[a.] $I = I_1 + (c) + I_2$, and both $I_1$ and $I_2$ are infinite without endpoints,
\item[b.] $(a_i)_{i\in I_1 + I_2}$ is $b$-indiscernible,
\end{itemize}
then
\[
\mathbb M \models \varphi(a_{i_1},\dots,a_{i_n};b) \leftrightarrow \varphi(a_{j_1},\dots,a_{j_n};b)
\]
for every $i_1<\dots < i_n$ and $j_1< \dots < j_n$ in $I$.
\end{defn}
\noindent This definition of distality of a single formula depends on the indicated partition of the free variables. It is immediate that $T$ is distal if and only if every $\Cal L$-formula is distal in $T$. Using saturation of $\mathbb M$, one can also see easily that in order to check the distality of a formula, one may assume that $I_1$ and $I_2$ are countable dense linear orders without endpoints.\newline

\noindent We now fix some notation. We will use $m,n$ for natural numbers. For $X\subseteq \mathbb M$, we shall write $\operatorname{dcl}_{\Cal L}(X)$ for the $\Cal L$-definable closure of $X$ in $\mathbb{M}$. When $T$ is an o-minimal theory, the closure operator $\operatorname{dcl}_{\Cal L}$ is a pregeometry. We will use this fact freely throughout this paper. For a tuple $b = (b_1,\ldots,b_n) \in \mathbb{M}^n$ and $X \subseteq \mathbb{M}$, by $Xb$ we mean $X \cup \{b_1,\ldots b_n\}$, and we say that $b$ is $\operatorname{dcl}_{\Cal L}$-independent if the set $\{b_1,\ldots, b_n\}$ is. For a function $f$, $\operatorname{ar}(f)$ will denote the arity of $f$.

\subsection*{Acknowledgements} We thank Pierre Simon and the anonymous referee for very helpful comments. We also thank Danul Gunatilleka, Tim Mercure, Richard Rast, Douglas Ulrich, and in particular Allen Gehret for reading an earlier version of this paper and providing us with excellent feedback.

\section{The discrete case}

In this section we give sufficient conditions for expansions of o-minimal theories by a function symbol to be distal. We prove in sections following this one that both tame pairs and the expansions by discrete groups mentioned above satisfy these conditions. This criterion for distality (and its proof) is closely related to the criterion for NIP given in \cite[Theorem 4.1]{GH-Dependent}. Here we use the same set up. As in \cite{GH-Dependent}, let $T$ be a complete o-minimal theory extending the theory of ordered abelian groups and let $\mathcal{L}$ be its language with distinguished positive element 1. Such a theory has definable Skolem functions. After extending it by constants and by definitions, we can assume the theory $T$ admits quantifier elimination and has a universal axiomatization. In this situation, any substructure of a model of $T$ is an elementary submodel, and therefore $\dcl_{\Cal L}(X) = \langle X \rangle$ for any subset $X$ of any model $\mathcal A$ of $T$; here $\langle X \rangle$ denotes the $\mathcal L$-substructure of $\mathcal A$ generated by $X$. For $\mathcal B \preceq \mathcal A$ we write $\mathcal B \langle X \rangle$ for $\langle \Cal B \cup X\rangle$. Following the notation from \cite{GH-Dependent} we extend $\mathcal{L}$ to $\mathcal{L}(\mathfrak f)$ by adding a new unary function symbol $\mathfrak f$. We let $T(\mathfrak f)$ be a complete $\mathcal{L}(\mathfrak f)$-theory
extending $T$. As usual, we take $\mathbb M$ to be a monster model of $T(\mathfrak f)$.
\begin{thm}\label{thmdiscpair}
Suppose that the following conditions hold:
\begin{itemize}
   \item [(i)] The theory $T(\mathfrak f)$ has quantifier elimination.
   \item [(ii)] For every $(\Cal C,\mathfrak f)\models T(\mathfrak f)$, $\Cal B\preceq
\Cal C$ with $\mathfrak f(\Cal B)\subseteq \Cal B$ and every $c\in \Cal C^k$, there are $l\in \N$ and $d \in  \mathfrak f\big(\Cal B\langle c\rangle\big)^l$ such that
\[
 \mathfrak f\big(\Cal B\langle c\rangle\big) \subseteq  \langle\mathfrak f(\Cal B),d \rangle.
\]
 \item [(iii)] Let $m\geq n$ and let $g,h$ be $\mathcal L$-terms of arities $m+k$ and $n+l$ respectively, $b_1\in \mathbb{M}^{k}, b_2 \in \mathfrak f(\mathbb{M})^{l}$, $(a_i)_{i \in I}$ be an indiscernible sequence from $ f(\mathbb{M})^{n}\times \mathbb{M}^{m-n}$ such that
 \begin{itemize}
 \item[a.] $ I=I_1 + (c) + I_2$, where both $I_1$ and $I_2$ are infinite and without endpoints, and $(a_i)_{i\in I_1 + I_2}$ is $b_1b_2$-indiscernible,
 \item[b.] $a_i=(a_{i,1},\dots,a_{i,m})$ for each $i\in I$, and
 \item[c.] $\mathfrak f(g(a_i,b_1))=h(a_{i,1},\dots,a_{i,n},b_2)$ for every $i \in I_1 + I_2$.
\end{itemize}
 Then $\mathfrak f(g(a_c,b_1))=h(a_{c,1},\dots,a_{c,n},b_2).$
\end{itemize}
Then $T(\mathfrak f)$ is distal.
\end{thm}
\begin{proof}
By (i), it is enough to show that every (partitioned) quantifier-free $\mathcal{L}(\mathfrak f)$-formula $\psi(x_1,\dots,x_p;y)$ is distal. We will prove this by induction on the number $e(\psi)$ of times $\mathfrak f$
occurs in $\psi$. If $e(\psi)=0$, this follows just from the fact that o-minimal theories are distal. Let $e\in \N_{>0}$ be such that every
quantifier-free $\mathcal{L}(\mathfrak f)$-formula $\psi'$ with $e(\psi')<e$ (with any partition) is distal.  Let $\psi(x_1,\dots,x_p;y)$ be a quantifier-free $\mathcal{L}(\mathfrak f)$-formula with $e(\psi)=e$. We will establish that $\psi$ is distal. Take an indiscernible sequence $(a_i)_{i \in I}$ from $\mathbb M^s$ and $b\in \mathbb M^{k}$ such that $ I=I_1 + (c) + I_2$, where both $I_1$ and $I_2$ are countable dense linear orders without endpoints, and $(a_i)_{i\in I_1 + I_2}$ is $b$-indiscernible. By $b$-indiscernibility we may assume that
\begin{itemize}
\item[(A)]  $\mathbb M\models \psi(a_{i_1},\dots,a_{i_p};b)$ for all $i_1 <\dots < i_p\in I_1+I_2$.
\end{itemize}
 Let $j \in \{1,\dots,p\}$, $u_1<\dots<u_{j-1} \in I_1$ and $v_{1} <\dots<v_{p-j}\in I_2$. It suffices to show that
\begin{equation}\label{eq:disc1}
\mathbb M \models \psi(a_{u_1},\dots,a_{u_{j-1}},a_c,a_{v_1},\dots,a_{v_{p-j}};b).
\end{equation}
Since $e>0$, there is an $\mathcal L$-term $g$ such that the term $\mathfrak f(g(x_1,\dots,x_p,y))$ occurs in
$\psi$.
Now let $\Cal A$ be the $\Cal L(\mathfrak f)$-substructure of $\mathbb M$ generated by $\{a_i:i \in I_1 + I_2\}$. By (ii), there is $d\in \mathfrak f\big(\Cal A\langle b\rangle\big)^{l}$ such that
\[\mathfrak f(\Cal A \langle b \rangle) \subseteq \langle \mathfrak f(\Cal A),d\rangle\]
(use $\mathbb M\restriction \mathcal L$ as $\mathcal C$ and $\mathcal A\restriction \mathcal L$ as $\mathcal B$ in the statement of (ii)).
Take $q,r\in \N$, $u_{j}<\dots <u_q\in I_1$ and $v_{-r}<\dots<v_0$ in $I_2$ such that
\begin{itemize}
\item[(B)] $u_1<\dots < u_q$ and $v_{-r}<\dots<v_{p-j}$,
\item[(C)] $d$ is in the $\Cal L(\mathfrak f)$-substructure generated by $a_u,a_v,b$,
\end{itemize}
where $a_u=(a_{u_1},\dots,a_{u_q})$ and $a_v=(a_{v_{-r}},\dots,a_{v_{p-j}})$. By the definition of $d$, we have for every $i \in I_1+I_2$
\begin{equation*}
\mathfrak f(g(a_{u_1},\dots,a_{u_{j-1}},a_i,a_{v_{1}},\dots,a_{v_{p-j}},b)) \in \langle \mathfrak f(\Cal A),d \rangle,
\end{equation*}
in particular when  $u_{q} < i < v_{-r}$.  Because $(a_i)_{i\in I_1+I_2}$ is $b$-indiscernible, we can (after possibly increasing $q,r$ and extending $a_u$ and $a_v$) find
  an $\mathcal L$-term $h$, $n\in \N$ and $\mathcal L(\mathfrak f)$-terms $t_1,\dots,t_n$ (all of the form $\mathfrak f(s_i)$ for an $\mathcal L$-term $s_i$) such that
\begin{itemize}
\item[(D)] for every $i \in I_1+I_2$ with $u_q < i < v_{-r}$
\[
\mathfrak f(g(a_{u_1},\dots,a_{u_{j-1}},a_i,a_{v_{1}},\dots,a_{v_{p-j}},b)) =h(t_1(a_{u},a_i,a_{v}),\dots,t_n(a_{u},a_i,a_{v}),d).
\]
\end{itemize}
Let $I'=(I_1)_{>u_q} + (c) + (I_2)_{<v_{-r}}$. For each $i\in I'$ set
\[
a_i' := (t_1(a_{u},a_i,a_{v}),\dots,t_n(a_{u},a_i,a_{v}),a_i).
\]
Since $(a_i)_{i\in I}$ is indiscernible, so is $(a_i')_{i \in I'}$. By (C) and the $b$-indiscernibility of $(a_i)_{i\in I_1+I_2}$, we get that $(a_i')_{i\in I'}$ is $bd$-indiscernible. Applying (iii) now using the indiscernible sequence $(a_i')_{i \in I'}$ and (D), we deduce that
\[
\mathfrak f(g(a_{u_1},\dots,a_{u_{j-1}},a_c,a_{v_{1}},\dots,a_{v_{p-j}},b)) =h(t_1(a_{u},a_c,a_{v}),\dots,t_n(a_{u},a_c,a_{v}),d).
\]
From this and (D), we get a quantifier-free $\mathcal{L}(\mathfrak f)$-formula $\psi'$ with $e(\psi')<e$ such that for all $i \in I'$
\begin{align*}
\mathbb{M} \models \psi(a_{u_1},&\dots,a_{u_{j-1}},a_i,a_{v_1},\dots,a_{v_{p-j}};b)\\
&\leftrightarrow\psi'\big(t_1(a_u,a_i,a_v),\dots,t_n(a_{u},a_i,a_{v}),a_i;b,d\big).
\end{align*}
By (A) and the induction hypothesis, \eqref{eq:disc1} follows.
\end{proof}

\section{Discrete groups}\label{discretegroups}

Let $\rtil$ be an o-minimal expansion of
$(\mathbb{R},<,+,\cdot,0,1)$ which is polynomially-bounded with
field of exponents $\mathbb{Q}$. We will establish distality for the theory of the expansion of $\rtil$ by a predicate for the cyclic multiplicative subgroup $2^{\mathbb{Z}}$ of $\R_{>0}$. Towards this goal, let $T$ be the theory of $\rtil$ and $\mathcal{L}$ be its language. Let $\lambda: \mathbb{R} \to
\R$ be the function that maps $x$ to $\max (-\infty, x] \cap 2^{\Z}$ when $x >0$, and to $0$ otherwise. It is immediate that the structures $(\rtil,2^{\mathbb{Z}})$ and $(\rtil,\lambda)$ define the same sets. Van den Dries \cite{vdd-Powers2} showed quantifier elimination for the latter structure when $\rtil$ is the real field. This result was generalized by Miller \cite{Miller-tame} to expansions of the real field with field of exponents $\Q$. It is worth pointing out that by \cite[Theorem 1.5]{discrete} the assumption on the field of exponents can not be dropped.\newline

\noindent  Let $T_{\operatorname{disc}}$ be the theory of
$(\rtil,\lambda)$ in the language $\mathcal{L}(\lambda)$, the extension of $\mathcal{L}$ by a unary function symbol for $\lambda$. In order to show distality of $T_{\operatorname{disc}}$, we can assume that $\rtil$ has
quantifier elimination and has a universal axiomatization.

\begin{thm}\label{discdependent} $T_{\operatorname{disc}}$ is distal.
\end{thm}
\begin{proof} We need to verify that $T_{\operatorname{disc}}$ satisfies
the assumptions of Theorem \ref{thmdiscpair}. Assumptions (i) and (ii) were already established in \cite[Theorem 6.5]{GH-Dependent}.
It is left to prove (iii).  Let $\mathbb{M}$ be a
monster model of $T_{\operatorname{disc}}$. We denote $\lambda(\mathbb{M})\setminus \{0\}$ by $G$. Note that $G$ is a multiplicative subgroup of $\mathbb{M}_{>0}$. For $p\in \N$, the set of $p$-powers $G^{[p]}:= \{g^p \:\ g \in G\}$ has finitely many cosets in $G$, since $|2^{\mathbb{Z}}: (2^{\mathbb{Z}})^{[p]}|=p$. Indeed, $1,2,\dots,2^{p-1}$ are representatives of the cosets of $G^{[p]}$.

\noindent Take an indiscernible sequence $(a_i)_{i \in I}$ from $\mathbb M^m$, where $I=I_1 + (c) + I_2$ and $I_1$ and $I_2$ are infinite without endpoints, such that
$a_{i,1},\dots,a_{i,n} \in \lambda(\mathbb{M})$ for every $i\in I$ and $a_i=(a_{i,1},\dots,a_{i,m})$. Let $(b_1,b_2)\in \mathbb M^{k} \times \lambda(\mathbb{M})^l$ such that $(a_i)_{i\in I_1 + I_2}$ is $b_1b_2$-indiscernible. Suppose that there are $\mathcal L$-terms $g,h$ such that for $i\in I_1 + I_2$
\[
\lambda(g(a_i,b_1))=h(a_{i,1},\dots,a_{i,n},b_2).
\]
It is left to conclude that $\lambda(g(a_c,b_1))=h(a_{c,1},\dots,a_{c,n},b_2)$. By definition of $\lambda$, we have for every $i \in I_1 + I_2$
\[
\mathbb{M} \models 1 \leq \frac{g(a_{i},b_1)}{h(a_{i,1},\dots,a_{i,n},b_2)}<2.
\]
Since $T$ is distal, the previous statement holds for all $i\in I$. It is left to show that $h(a_{c,1},\dots,a_{c,n},b_2) \in G$.
By \cite[Corollary 6.4]{GH-Dependent} and $b$-indiscernibility of $(a_i)_{i\in I_1 + I_2}$, there are $t,q_1,\dots,q_n\in \mathbb{Q}, r=(r_1,\dots,r_l) \in \mathbb{Q}^l$ such that for every $i \in I_1+I_2$
\[
h(a_{i,1},\dots, a_{i,n},b_2) = 2^t \cdot
a_{i,1}^{q_1}\cdots a_{i,n}^{q_n}\cdot b_2^{r},
\]
where $b_2^{r}$ stands for $b_{2,1}^{r_1}\cdots b_{2,l}^{r_l}$. By distality of $T$, this equation holds for all $i\in I$.
\noindent It is left to show that $ 2^t \cdot a_{c,1}^{q_1}\cdots a_{c,n}^{q_n}\cdot b_2^{r} \in G$. Let $p\in \N$ be such that $p\cdot t,p\cdot q_1,\dots,p \cdot q_n
\in \mathbb{Z}$ and $p \cdot r \in \mathbb{Z}^{l}$. It is enough to prove $ 2^{p\cdot t} \cdot a_{c,1}^{p \cdot q_1}\cdots a_{c,n}^{p \cdot q_n} \in b_2^{p\cdot r} \cdot G^{[p]}$. Let $s \in
\{0,\dots,p-1\}$ be such that $b_2^{p \cdot r}$ is in $2^s \cdot
G^{[p]}$. Then for every $i \in I$,
\begin{equation}\label{thmdiscgroupeq}
2^{p\cdot t} \cdot a_{i,1}^{p \cdot q_1}\cdots a_{i,n}^{p\cdot q_n} \in b_2^{p\cdot r} \cdot G^{[p]} \hbox{ iff } 2^{p\cdot t} \cdot a_{i,1}^{p \cdot q_1}\cdots a_{i,n}^{p \cdot q_n} \in 2^s \cdot G^{[p]}.
\end{equation}
Since the second statement in \eqref{thmdiscgroupeq} holds for $i \in I_1+I_2$ and $(a_i)_{i\in
I}$ is indiscernible, it holds for all $i \in I$ and in particular for $i=c$. \end{proof}

\section{Tame pairs}
For this section, let $T$ be a complete o-minimal theory expanding the theory of real closed
fields in a language $\mathcal{L}$. In \cite{vdDL-T-convexity} van den Dries and Lewenberg introduced the following notion of tame pairs of o-minimal structures.

\begin{defn}
A pair $(\Cal A,\Cal B)$ of models of $T$ is called a {\it tame pair} if $\Cal B\preceq \Cal A$, $\Cal A\neq \Cal B$ and for every $a \in \Cal A$ which is in the convex hull of
$\Cal B$, there is a unique $\operatorname{st}(a) \in \Cal B$ such that $|a-\operatorname{st}(a)|<b$ for
all $b \in \Cal B_{>0}$.
\end{defn}

\noindent The \emph{standard part map} $\operatorname{st}$ can be
extended to all of $\Cal A$ by setting $\operatorname{st}(a)=0$ for all $a$ not in the
convex hull of $\Cal B$. Instead of considering $(\Cal A, \Cal B)$ we will consider $(\Cal A,\operatorname{st})$. It is easy to check that these two structures are interdefinable.
Let $T_t$ be the $\Cal L(\operatorname{st})$-theory of all structures of the form $(\Cal A,\operatorname{st}).$ After extending $T$ by
definitions, we can assume that $T$ has quantifier elimination and is universally axiomatizable. By \cite[Theorem 5.9]{vdDL-T-convexity} and \cite[Corollary 5.10]{vdDL-T-convexity}, $T_{t}$ is complete and has quantifier
elimination.\newline

\noindent We will also need to consider the theory of convex pairs. A $T$-{\it convex} subring of a model $\Cal A$ of $T$ is a convex subring that is closed under all continuous unary $\Cal L$-$\emptyset$-definable functions.
A \emph{convex pair} is a pair $(\Cal A,V)$, where $\Cal A\models T$, $V$ is a $T$-convex subring of $A$, and $V \neq A$. We denote the theory of all such pairs by $T_c$. By  \cite[Corollary 3.14]{vdDL-T-convexity}, this theory is weakly o-minimal. By \cite[Theorem 4.1]{DGL-DP}, every weakly o-minimal theory is dp-minimal and hence distal by \cite[Lemma 2.10]{Simon-Distal}. Therefore $T_c$ is distal.\newline

\noindent For every model $(\Cal A,\operatorname{st})$ of $T_t$, the pair $(\Cal A,V)$ is a model of $T_c$, where $V$ is
the convex closure of $\operatorname{st}(\Cal A)$. It follows immediately that for every $b \in \operatorname{st}(\Cal A)$ and $a
\in \Cal A$
\[
 \operatorname{st}(a)=b \quad \Longleftrightarrow \quad a=b \textrm{ or } ((a-b)^{-1} \notin V) \textrm { or } (b=0 \textrm{ and } a \notin V).
\]
We will not use the explicit description on the right, but we will use the fact that this gives us an $\Cal L(U)$-formula $\psi$ such that for all $a\in \Cal A$ and $b \in \operatorname{st}(\Cal A)$
\begin{equation}\label{convexeq}
(\Cal A,\operatorname{st}) \models  \operatorname{st}(a)=b \hbox{ iff } (\Cal A,V)\models \psi(a,b).
\end{equation}

\begin{thm}\label{tamedistal} $T_t$ is distal.
\end{thm}
\begin{proof} We will show that $T_t$ satisfies the
assumptions of Theorem \ref{thmdiscpair}. Assumptions (i) and (ii) were already established for \cite[Theorem 5.2]{GH-Dependent}.
We only need to prove (iii). Let $\mathbb{M}$ be a
monster model of $T_t$, and $V$ the convex closure of $\operatorname{st}(\mathbb{M})$. Let $(a_i)_{i \in I}$ be an indiscernible sequence from $\mathbb M^m$, where $I=I_1 + (c) + I_2$ and $I_1$ and $I_2$ are infinite with no endpoints, such that
$a_{i,1},\dots,a_{i,n} \in \operatorname{st}(\mathbb{M})$ for $i\in I$ and $a_i=(a_{i,1},\dots,a_{i,m})$. Let $(b_1,b_2)\in \mathbb M^{k}\times \operatorname{st}(\mathbb{M})^l$ such that $(a_i)_{i\in I_1 + I_2}$ is $b_1b_2$-indiscernible. Suppose that there are $\mathcal L$-terms $g,h$ such that for $i\in I_1 + I_2$
\[
\operatorname{st}(g(a_i,b_1))=h(a_{i,1},\dots,a_{i,n},b_2).
\]
We need to show that $\operatorname{st}(g(a_c,b_1))=h(a_{c,1},\dots,a_{c,n},b_2)$.
Since $\operatorname{st}(\mathbb{M})$ is a model of $T$, we have $h(a_{i,1},\dots ,a_{i,n},b_2) \in \operatorname{st}(\mathbb{M})$ for every $i\in I$.
By \eqref{convexeq}, there is an $\Cal L(U)$-formula $\psi$ such that for $i \in I$
\[
(\mathbb{M},\operatorname{st})\models \operatorname{st}(g(a_i,b_1))=h(a_{i,1},\dots, a_{i,n},b_2)\Longleftrightarrow (\mathbb{M},V)\models \psi(a_i,b).
\]
Since $T_c$ is distal, $\mathbb{M} \models \psi(a_c,b)$.
\end{proof}

\section{Dense Pairs}

In this section we present sufficient conditions for non-distality of expansions of o-minimal theories by a single unary predicate, and give several examples of NIP theories satisfying these conditions. Let $T$ be an o-minimal theory in a language $\mathcal{L}$ expanding that of ordered abelian groups, $U$ a unary relation symbol not appearing in $\mathcal{L}$, and $T_U$ an $\mathcal{L}(U)=\mathcal{L}\cup \{U\}$-theory expanding $T$. Let $\mathbb M$ be a monster model of $T_U$. We denote the interpretation of $U$ in $\mathbb{M}$ by $U(\mathbb{M})$. We say that an $\Cal L(U)$-definable subset $X$ of $\mathbb{M}$ is \emph{small} if there is no $\Cal L$-definable (possibly with parameters) function $f: \mathbb M^m \to \mathbb M$ such that $f(X^m)$ contains an open interval in $\mathbb M$. When we say a set is dense in $\mathbb M$, we mean dense with respect to the usual order topology on $\mathbb M$.

\begin{thm}\label{thm:nondistal}
Suppose the following conditions hold:

\begin{enumerate}

\item $U(\mathbb{M})$ is small and dense in $\mathbb{M}$.

\item For $n \in \N$, $C \subseteq \mathbb{M}$, and $a,b \in \mathbb{M}^n$ both $\dcl_{\Cal L}$-independent over $C \cup U(\mathbb{M})$, $$tp_\Cal{L}(a|C) = tp_\Cal{L}(b|C)\Rightarrow tp_{\Cal L(U)} (a|C)=tp_{\Cal L(U)} (b|C).$$


\end{enumerate}

Then $T_U$ is not distal.

\end{thm}

\begin{proof}
Let $b\in \mathbb{M}$ be $\operatorname{dcl}_{\Cal L}$-independent over $U(\mathbb{M})$. The existence of such a $b$ follows immediately from smallness of $U(\mathbb{M})$ and saturation of $\mathbb{M}$. Let $I_1,I_2$ be two countable linear orders without endpoints. Consider a set $\Phi$ containing $\Cal L(U)$-$b$-formulas in the variables $(x_i)_{i\in I_1+(c)+I_2}$ expressing the following statements:

\begin{enumerate}

\item[(i)] $\{x_i:i\in I_1+I_2\}$ is $\operatorname{dcl}_{\Cal L}$-independent over $U(\mathbb{M})b$,

\item[(ii)]  $f(x_{i_1},\ldots,x_{i_n},b)<x_{i_{n+1}}$, for each $i_1<\dots < i_{n+1} \in I_1+(c)+I_2$ and $\Cal L$-$\emptyset$-definable function $f$,


\item[(iii)] there is $u \in U(\mathbb{M})$ such that $x_c = u+b$. 

\end{enumerate}
We will show that $\Phi$ is realized in $\mathbb{M}$. By saturation of $\mathbb{M}$ it is enough to show that every finite subset $\Phi_0$ of $\Phi$ is realized. Let $\mathcal{F}=\{f_1,\ldots,f_m\}$ be the $\Cal L$-definable functions appearing in formulas of the form (ii) in $\Phi_0$. Let $i_1<\dots<i_n \in I_1+(c)+I_2$ be the indices of variables occurring in $\Phi_0$. We may assume $c$ is among these, and by adding dummy variables that each $f_j$ is of the form $f(x_{i_1},\ldots x_{i_k},b)$ for some $k<n$. We now recursively choose  $(a_{i_1},\ldots,a_{i_n})$ realizing the type $\Phi_0$.


\noindent Suppose we have defined $a_{i_1},\ldots, a_{i_{k-1}}$. If $k=1$, we will have defined no previous $a_i$, and the functions below will be of arity $1$ only mentioning $b$. If $i_k=c$, then by denseness of $U(\mathbb{M})$ we may choose $a_c$ in
\[
\big(b+U\big) \cap \Big(\max_{f\in \mathcal{F}, ar(f)=k} f(a_{i_1},\ldots, a_{i_{k-1}},b),\infty\Big).
\]
\noindent If $i_k\neq c$, then by smallness of $U(\mathbb{M})$ we may choose $a_{i_k}$ in
\[
\Big(\max_{f\in \mathcal{F}, ar(f)=k} f(a_{i_1},\ldots, a_{i_{k-1}},b),\infty\Big) \setminus \operatorname{dcl}_{\Cal L}(U(\mathbb{M})ba_{i_1}\cdots a_{i_{k-1}}).
\]
As $(a_{i_1},\ldots, a_{i_n})$ realizes $\Phi_0$, $\Phi$ is finitely satisfiable.




\noindent By saturation, we can pick a realization $(a_i)_{i \in I_1+(c)+I_2}$ of $\Phi$ in $\mathbb{M}$. This sequence can be thought of as a very rapidly growing sequence; each element will realize the type at $+\infty$
over the $\Cal L$-definable closure of everything before it. Therefore the sequence is a Morley sequence for the $\Cal L$-type of $+\infty$ over $\dcl_{\Cal L}(b)$, and hence is $\Cal L$-$b$-indiscernible. As $\dcl_\Cal{L}$ is a pregeometry and $b$ is $\dcl_{\Cal L}$-independent over $U(\mathbb M)$, (i) and (iii) together imply that the full sequence is $\dcl_\Cal{L}$-independent over $U(\mathbb M)$. Thus by (2), the $\Cal L$-indiscernibility of these sequences lifts to $\Cal L(U)$-indiscernibility; that is, $(a_i)_{i \in I_1+I_2}$ is $\Cal{L}(U)$-indiscernible over $b$, and the full sequence is $\Cal{L}(U)$-indiscernible. However, since $a_c = b+u$ for some $u \in U(\mathbb{M})$, the full sequence is not $\Cal{L}(U)$-indiscernible over $b$. Hence $T_U$ is not distal.
\end{proof}

\subsection*{Optimality} Note that the assumption that $T$ expands the theory of ordered abelian groups can not be dropped. As pointed out in the introduction the theory of the structure $(\R,<,\Q)$ is distal. However, it is not hard to check that the theory of $(\R,<,\Q)$ satisfies the other assumptions of Theorem \ref{thm:nondistal}.

\subsection*{Dense pairs} Let $\Cal A,\Cal B$ be two models of an o-minimal theory $T$ expanding the theory of ordered abelian groups such that $\Cal B \preceq \Cal A$, $\Cal B\neq \Cal A$, and $\Cal B$ is dense in $\Cal A$. We call $(\Cal A,\Cal B)$ a \emph{dense pair} of models of $T$. Let $T^d$ be the theory of dense pairs in the language $\Cal L(U)$. By van den Dries \cite{densepairs} $T^d$ is complete. Moreover, for every dense pair $(\Cal A,\Cal B)$, the underlying set of $\Cal B$ is small by \cite[Lemma 4.1]{densepairs}. While not stated explicitly, it follows almost immediately from \cite[Claim on p.67]{densepairs} that $T^d$ also satisfies (2) of Theorem \ref{thm:nondistal} (see \cite[Proposition 2.3]{EGH} for detailed proof). Therefore $T^d$ is not distal.

\subsection*{Dense groups} Let $\OR$ be the real field $(\R,<,+,\cdot,0,1)$. Let $\Gamma$ be a dense subgroup of $\mathbb{R}_{>0}$ that has the \emph{Mann property}, that is for every $a_1,\dots,a_n \in \mathbb{Q}^{\times}$, there are finitely many $(\gamma_1,\dots,\gamma_n) \in \Gamma^n$ such that $a_1\gamma_1+\dots+a_n\gamma_n = 1$ and $\sum_{i\in I} a_i\gamma_i \neq 0$ for every proper nonempty subset $I$ of $\{1,\dots,n\}$. Every multiplicative subgroup of finite rank in $\R_{>0}$ has the Mann property. Let $\Cal L$ be the language of $\OR$ expanded by a constant symbol for each $\gamma \in \Gamma$. Let $T_{\Gamma}$ be the $\Cal L(U)$-theory of $(\OR,(\gamma)_{\gamma \in \Gamma},\Gamma)$ in this language. This structure was studied in detail by van den Dries and G\"unayd\i n \cite{densegroups}. A proof that every model satisfies (1) of Theorem \ref{thm:nondistal} is in \cite[Proposition 3.5]{GH-Elliptic}. Similarly to dense pairs, it is not mentioned in \cite{densegroups} that these theories satisfy condition (2) of Theorem \ref{thm:nondistal}. However, it can easily be deduced from the proof of \cite[Theorem 7.1]{densepairs} (see also  \cite[p. 6]{EGH}).

\noindent The argument can easily be extended to related structures (see \cite{BZ,GH-Elliptic,H-tau}).

\subsection*{Independent sets} We finish with another class of structures that were studied recently by Dolich, Miller and Steinhorn \cite{dms2}. Let $T$ be an o-minimal theory in a language $\mathcal{L}$ expanding that of ordered abelian groups. Let $T^{\operatorname{indep}}$ be an $\Cal L(U)$-theory extending $T$ by axioms stating that $U$ is dense and $\dcl_{\Cal L}$-independent. By \cite{dms2}, $T^{\textrm{indep}}$ is complete. Every model of $T^{\textrm{indep}}$ satisfies (1) of Theorem \ref{thm:nondistal} by \cite[2.1]{dms2}. By \cite[2.12]{dms2}, condition (2) of that theorem also holds for $T^{\operatorname{indep}}$.

  \bibliographystyle{plain}
  \bibliography{hieronymi}

\end{document}